\documentclass[12pt]{amsart}

\usepackage{amsmath,amsfonts,amsthm,amsopn,cite,mathrsfs, amssymb}
\usepackage{epsfig,verbatim,color}
\usepackage{enumerate}
\usepackage{subcaption}
\usepackage{hyperref}
\hypersetup{
	colorlinks   = true, 
	urlcolor     = blue, 
	linkcolor    = blue, 
	citecolor   = blue 
}

\def\supp{\operatorname{supp}}

\newcommand{\CC}{{\mathbb{C}}}

\numberwithin{equation}{section}

\newcommand\lsdots{\hbox to 0.75em{.\hss.\hss.}}
\newcommand\lssdots{\hbox to 0.45em{.\hss.\hss.}}
\newcommand\csdots{\hbox to 0.75em{\ensuremath{\cdot}\hss\ensuremath{\cdot}\hss\ensuremath{\cdot}}}

\def\var{\mathrm{Var}}
\def\tr{\mathrm{Tr}}

\newcommand{\Ti}{{T^1_k}}
\newcommand{\Tii}{{T^2_k}}

\newcommand{\rk}{{T^*_k}}

\newtheorem{theorem}{Theorem}[section]
\newtheorem{lemma}[theorem]{Lemma}
\newtheorem{prop}[theorem]{Proposition}
\newtheorem{cor}[theorem]{Corollary}

\newtheorem{rem}[theorem]{Remark}

\newcommand{\der}{\nabla}

\usepackage[
textwidth=17.0cm,
textheight=22cm,
hmarginratio=1:1,
vmarginratio=1:1]{geometry}
\author[A. Haimi]{Antti Haimi}
\email{antti.haimi@univie.ac.at}
\address{Faculty of Mathematics, University of Vienna,
Oskar-Morgenstern-Platz 1, A-1090 Vienna, Austria}

\author[J. L. Romero]{Jos\'{e} Luis Romero}
\email{jose.luis.romero@univie.ac.at}
\address{Faculty of Mathematics, University of Vienna,
Oskar-Morgenstern-Platz 1, A-1090 Vienna, Austria\\
and
Acoustics Research Institute, Austrian Academy of Sciences, Wohllebengasse
12-14 A-1040, Vienna, Austria}
\thanks{ A.\ H.\ and J.\ L.\ R.\ gratefully acknowledge support from the Austrian Science Fund (FWF): Y 1199.}
\begin{document}
\title{Normality of smooth statistics for planar determinantal point processes}

\keywords{asymptotic normality, determinantal point process, linear statistics, Weyl-Heisenberg DPP}
\subjclass[2020]{60F05, 60F17, 30H05, 60G55, 60B20, 82B31,46C07}

\begin{abstract}
We consider smooth linear statistics of determinantal point processes on the complex plane, and their large scale asymptotics. We prove asymptotic normality in the finite variance case, where Soshnikov's theorem is not applicable.
The setting is similar to that of Rider and Vir\'ag 
[Electron. J. Probab., 12, no. 45, 1238--1257, (2007)]
for the complex plane, but replaces analyticity conditions by the
assumption that the correlation kernel is reproducing. Our proof is a streamlined version of that of Ameur, Hedenmalm and Makarov [Duke Math J., 159, 31--81, (2011)]
for eigenvalues of normal random matrices. In our case, the reproducing property is brought to bear to compensate for the lack of analyticity and radial symmetries.
\end{abstract}

\maketitle
\section{Introduction}
\subsection{Determinantal point processes}
Determinantal point processes (DPP's) are certain interacting particle systems which exhibit repulsion between different points
\cite{soshnikov2000determinantal, MR2932631, MR2581882, MR2552864}. Their central feature is that all the statistical information is encoded in a single function, called correlation kernel. Since their introduction by Macchi in 1975 \cite{macchi1975coincidence}, DDP's have found extensive applications in areas of knowledge as varied as random matrix theory, number theory, quantum mechanics and machine learning.

For a precise definition, we restrict attention to the Euclidean space and consider a random subset $\mathcal{X} \subset \mathbb{R}^d$ encoded as a random integer-valued positive Radon measure $E \mapsto \mathcal{X}(E)$, which counts the number of points within a Borel set. We only consider \emph{simple} processes, that is, almost surely $\mathcal{X}$ assigns mass at most one to each singleton $\{x\}$. Such a process is said to admit \emph{correlation functions} (with respect to the Lebesgue measure) if for each $n \in \mathbb{N}$, there exists a function $\rho_n: (\mathbb{R}^d)^n \to [0, \infty)$ such that
for all mutually disjoint Borel sets $E_1,\ldots,E_n \subset \mathbb{R}^d$,
$$
\mathbb{E} \Big[\prod_{j=1}^n X(E_j) \Big] = \int_{E_1 \times \cdots \times E_n} \rho_n(x_1,\ldots, x_n) \,dx_1 \ldots dx_n.
$$
In addition, $\rho_n(x_1,\ldots,x_n)$ is required to vanish if $x_j=x_{k}$ for some $j\not=k$. The correlation functions $\rho_k$, also called \emph{joint intensities}, describe the essential statistical properties of $\mathcal{X}$. A simple point process $\mathcal{X}$ is called \emph{determinantal} if there exists a function $K: \mathbb{R}^d \times \mathbb{R}^d \to \mathbb{C}$, called \emph{correlation kernel}, such that
$$
\rho_n(x_1,\ldots, x_n) = \det K(x_k,x_j)_{k,j=1,\ldots, n}, \qquad x_1, \ldots, x_n \in \mathbb{R}^d.
$$
One of the most famous DDP's is the infinite  \emph{Ginibre ensemble} \cite{MR173726} on $\mathbb{C}^r\simeq\mathbb{R}^{2r}$, which has correlation kernel
\begin{align}\label{eq_gin_d}
K(z,w)= \frac{1}{\pi^r} \exp\Big[{\sum_{j=1}^r z_j \bar w_j-\frac12 |z_j|^2-\frac12|w_j|^2}\Big], \qquad z,w \in \mathbb{C}^r.
\end{align}
The Macchi-Soshnikov theorem \cite{macchi1975coincidence} provides a machinery to produce examples of DPP's: any Hermitian symmetric kernel $K: \mathbb{R}^d \times \mathbb{R}^d \to \mathbb{C}$ which induces a locally trace class integral operator with spectrum in the interval $[0,1]$ is the correlation kernel of a uniquely determined DPP. A simple way to satisfy the spectral assumption is to let $\mathcal{H} \subset L^2(\mathbb{R}^d)$ be a
\emph{reproducing kernel subspace} - that is, a closed subspace where evaluations are continuous functionals - and let $K$ be 
its \emph{reproducing kernel}, that is, the integral kernel representing the orthogonal projection $L^2(\mathbb{R}^d) \to \mathcal{H}$. As such, the kernel satisfies the \emph{reproducing formula}
\begin{align}\label{eq_rf}
K(x,y) = \int K(x,u) K(u,y) \, du.
\end{align}
The Ginibre kernel \eqref{eq_gin_d} in one complex dimension $r=1$ corresponds in this way to the weighted Bargmann-Fock space $\mathcal{H}=\mathcal{F}_0(\mathbb{C})$,
\begin{align}\label{eq_f0}
\mathcal{F}_1(\mathbb{C}) = \Big\{ F(z) e^{-|z|^2/2}: F:\CC\to\CC \text{ entire, } \int_{\mathbb{C}} |F(z)|^2e^{-|z|^2}dA(z) < \infty\Big\},
\end{align}
where $dA$ is the Lebesgue (area) measure. Further examples, important in high energy physics, are the polyanalytic Ginibre ensembles \cite{haimi2013polyanalytic, shirai2015ginibre} associated with the weighted poly-Bargmann spaces \cite{vas00}
$\mathcal{H}=\mathcal{F}_k(\mathbb{C})$,
\begin{align}\label{eq_fk}
\mathcal{F}_k(\mathbb{C}) = \Big\{ F(z) e^{-|z|^2/2}: \bar{\partial}^k F = 0, \int_{\mathbb{C}} |F(z)|^2e^{-|z|^2}dA(z) < \infty\Big\},
\end{align}
or with the pure poly-Bargmann spaces $\mathcal{F}_k(\mathbb{C}) \ominus \mathcal{F}_{k-1}(\mathbb{C})$ \cite{vas00}.

\subsection{Linear statistics and asymptotic normality}
One of the main topics of interest around DPP's has been the validity of central limit theorems for various statistical observations. In this article we will be interested in the \emph{linear statistic}
\begin{align*}
\tr(f) := \sum_{x \in \mathcal{X}} f(x)
\end{align*}
associated with a \emph{smooth} function $f: \mathbb{R}^d \to \mathbb{C}$ and a DPP $\mathcal{X}$. Specifically, we will consider large scale asymptotics of linear statistics
\begin{align}\label{eq_trr}
\tr_\rho(f) = \sum_{x \in\mathcal{X}_\rho} f(x)
\end{align}
calculated under increasing contractions of the point process
\begin{align}\label{eq_dil}
\mathcal{X}_\rho = \big\{ \tfrac{1}{\sqrt{\rho}} x: x \in \mathcal{X} \big\},
\qquad \rho \to \infty,
\end{align}
or, more generally, under qualitatively similar asymptotic deformations.
The corresponding expectations are given by
\begin{align*}
\mathbb{E}\big[\tr_\rho(f)\big] = \int_{\mathbb{R}^d} f(x) K_\rho(x,x)\,dx,
\end{align*}
where $K_\rho$ is the correlation kernel of $\mathcal{X}_\rho$. In the model case \eqref{eq_dil},
\begin{align}\label{eq_e}
\mathbb{E}\big[\tr_\rho(f)\big] = \rho^{d/2} \int_{\mathbb{R}^d} f(x) K(\sqrt{\rho} x, \sqrt{\rho} x) \,dx,
\end{align}
where $K$ is the correlation kernel of $\mathcal{X}$.

Soshnikov's celebrated central limit theorem \cite[Theorem 1]{soshnikov2002gaussian}, which builds on \cite{costin1995gaussian}, gives very general conditions under which the standardized linear statistic
\begin{align}\label{eq_st}
\widetilde{\tr_\rho}(f) :=
\frac{ \tr_\rho (f) - \mathbb{E} \big[\tr_\rho (f)\big]}{\sqrt{ \var \big[\tr_\rho (f)}\big]}
\end{align}
converges in distribution to a normal variable. The key assumption is that the expected value and variance of $\tr_\rho(f)$ be related by
\begin{equation} \label{eq:variance_cond}
\mathbb{E} \big[\tr_\rho |f|\big] = 
\mathrm{O} \big(\var [\tr_\rho (f)] ^\delta\big),
\mbox{ as } \rho \to \infty,
\end{equation}
for some $\delta>0$, while $\var [\tr_\rho (f)] \to \infty$.

Soshnikov's theorem immediately implies, for example, that the standardized point counting statistics for the Ginibre ensemble \eqref{eq_gin_d}, where $f$ is the indicator function of a smooth domain, are asymptotically normal. For smooth statistics, the Ginibre ensemble \eqref{eq_gin_d} satisfies
the key relation \eqref{eq:variance_cond} only in complex dimension $r>1$. While standardized smooth linear statistics of the Ginibre ensemble are also asymptotically normal in one complex dimension, the planar case requires ad hoc arguments \cite{MR2346510}.

\subsection{Asymptotic normality for planar DPP's}
Rider and Vir\'ag introduced an abstract model to study the normality of standardized smooth statistics of DDP's on certain Riemann surfaces \cite{MR2346510}. The central assumption is an approximate form of analyticity of the correlation kernel, and applies to the planar Ginibre ensemble, as well as to its analogues on the sphere and unit disk.

A second technique was introduced by Ameur, Hedenmalm and Makarov to study the normality of fluctuations of eigenvalues of certain ensembles of random matrices under a confining potential of increasing strength \cite{ameur2011fluctuations, ameur2010berezin} (see also Berman's work in several complex variables \cite{MR3903320}). While the corresponding correlation kernels are analytic, they do not have an explicit formula that would facilitate the approach of \cite{MR2346510, rider2007noise}. 

Beyond the analytic setting, normality of fluctuations of smooth statistics was shown to hold for certain planar DDPs in \cite{haimi2016central}. Here, the correlation kernel is \emph{polyanalytic} and the key insight is that the cumulants of \eqref{eq_trr} can be related by precise algebraic formulas to those of the analytic Ginibre ensemble, essentially reducing the problem to \cite{rider2007noise}.

The first goal of this article is to present a general result on asymptotic normality of smooth statistics for planar DPP's with correlation kernels that satisfy the \emph{reproducing property} \eqref{eq_rf}. 
The result we derive neither requires analyticity nor relies on algebraic relations and symmetries linking the correlation kernel to the Ginibre one. Our main insight is that the technique from \cite{ameur2011fluctuations} can be streamlined to 
substitute many of the cancellations attributed to analyticity and radiality to instead rely on the reproducing property. The reproducing assumption is to a great extent optimal, since under mild assumptions non-reproducing kernels fall within the scope of Soshnikov's theorem (see Sections \ref{sec_res} and \ref{sec_b}).

The result we present can be applied for example to infinite poly-analytic Ginibre ensembles considered with respect to general potentials, that is, infinite analogues of the finite ensembles introduced in \cite{haimi2013polyanalytic}. These are related to Hilbert spaces analogous to \eqref{eq_fk} but with respect to general weights
$Q: \CC\to [0,\infty)$:
\begin{align}\label{eq_m}
	\mathbb{H} = \Big\{ F(z) e^{-Q(z)/2}: \bar{\partial}^k F = 0, \int_{\mathbb{C}} |F(z)|^2e^{-Q(z)}\,dA(z) < \infty\Big\},
\end{align}
and cannot be related to their analytic counterparts by simple algebraic formulae, as done for standard potentials in \cite{haimi2016central}.

A second motivating application is to the kernels arising from the Schr\"odinger representation of the Weyl-Heisenberg group. Let $g: \mathbb{R} \to \mathbb{C}$ be a smooth and fast decaying function satisfying $\int_{\CC} |g|^2\,dA=1$. Then
\begin{align}\label{eq_wh}
K_g(z,w) = \int_{\mathbb{R}} \overline{g(t-x_1)}g(t-x_2) e^{2\pi i t (y_2-y_1)} \,dt, \qquad z=x_1 + i y_1, w= x_2+iy_2 \in \mathbb{C},
\end{align}
is a locally trace class reproducing kernel acting on $L^2(\CC,dA)$
\cite[Chapters 1 and 2]{folland89}. Non-smooth linear statistics for the associated DPP
have been considered in \cite{abreu2016weyl, matsui2021local}, providing variance asymptotics. Under mild hypotheses, non-smooth statistics satisfy \eqref{eq:variance_cond}, and, as with the Ginibre ensemble, Soshnikov's theorem readily shows that standardized statistics corresponding to the number of points within a growing disk are asymptotically normal (while the investigation of precise rates for such convergence is more challenging \cite{MR4418712}). As an application of our main result, we shall conclude that Weyl-Heisenberg DPP's \eqref{eq_wh} also enjoy asymptotically normal \emph{smooth} linear statistics. Interestingly, the kernel \eqref{eq_wh} may exhibit no radial symmetries. 

\section{Results}\label{sec_res}
Let $\mathcal{H}_\rho \subset L^2(\CC, dA)$, $\rho>0$, be a family of reproducing Hilbert spaces, where $dA$ is the Lebesgue area measure, and denote by $K_\rho: \CC \times \CC \to \mathbb{C}$ the corresponding reproducing kernels. We will assume that the kernels $K_\rho$ satisfy the following: there exist continuous \emph{envelope functions}
$\phi_\rho:\CC \to [0,\infty)$ such that
\begin{align}\label{eq_env}
|K_{\rho}(z,w)| \leq \phi_{\rho}(z-w), \qquad z,w\in\CC.
\end{align}
The envelops are further assumed to satisfy the following properties:
\begin{enumerate}
	\item[$\bullet$] (\emph{Size})	
	\begin{align}\label{A1}
	\sup_{\rho>0} \sup_{z\in\CC} \tfrac{1}{\rho} \phi_\rho(z) < \infty,
	\end{align}
	\item[$\bullet$] (\emph{Uniform integrability})
	\begin{align}\label{A2}
	\sup_{\rho>0} \int_{\CC} \phi_\rho \,dA <\infty,
	\end{align}
	\item[$\bullet$] (\emph{Interaction decay}) 
	\begin{align}\label{A3}
	\lim_{\rho\to\infty} \rho \int_{\CC} |z|^3 \phi_\rho(z)dA(z) = 0.
	\end{align}
\end{enumerate}
The above assumptions mean that the envelops behave qualitatively like dilations $\phi_\rho(z) = \rho \phi(\sqrt{\rho}z)$.
The model is essentially as in \cite{MR2346510}, but, crucially, does not assume analyticity. On the other hand, here $K_\rho$ is assumed to be reproducing, while 
\cite{MR2346510} requires an asymptotic reproducing property for polynomials.
\begin{rem}\label{rem_ltc}
A positive kernel $K_\rho$ with envelope $\phi_\rho$ satisfying \eqref{A1} is \emph{locally trace class}, that is, for every compact set $B \subset \mathbb{C}$, the localized kernel $K^{B}_\rho(z,w)=1_B(z) K_\rho(z,w) 1_B(w)$ represents a trace-class integral operator. Indeed,
the operator with integral kernel $K^{B}_\rho$ is positive because it is a compression of that with kernel $K_\rho$, while the corresponding trace is $\lesssim \rho |B|$ by \eqref{eq_env} --- if $K$ is continuous this follows from
$\int_\mathbb{C} K^{B}_\rho(z,z) \,dA(z) = \int_{B} K_\rho(z,z) \,dA(z) \lesssim \rho |B|$,
see, e.g., \cite[Theorem 2.1]{simon}, while in general one can use an approximation argument. Hence, 
if $K_\rho$ is also reproducing, by the Soshnikov-Macchi theorem \cite{macchi1975coincidence} \cite[Theorem 3]{soshnikov2000determinantal}, it is the correlation kernel of a unique determinantal point process.
\end{rem}
The following is our main result on asymptotic normality.
\begin{theorem} \label{thm:main}
	Let $\{K_\rho: \rho>0\}$ be a family of Hermitian symmetric reproducing kernels acting on $L^2(\CC, dA)$ 
	with continuous envelopes $\eqref{eq_env}$ satisfying \eqref{A1}, \eqref{A2}, \eqref{A3}. Let $f$ be a compactly supported, real-valued $C^3$ test function. Then the cumulants of the order $k \geq 3$ of the variable $\tr_\rho(f)$ tend to zero as $\rho \rightarrow \infty$. As a consequence: 
	\begin{itemize}
	\item[(i)] If  $\var \big[\tr_\rho(f)\big] \to \sigma^2$
	for some finite $\sigma \geq 0$, then $\tr_\rho(f)-\mathbb{E}[\tr_\rho(f)] \to \mathcal{N}(0,\sigma^2)$ in distribution;
	\item[(ii)] If $\liminf_{\rho\to\infty} \var \big[\tr_\rho(f)\big] > 0$, then $\widetilde{\tr_\rho(}f) \to \mathcal{N}(0,1)$ in distribution.
	\end{itemize}
(In (i), the degenerate case $\sigma=0$ means that the limit variable is almost surely zero.)
	\end{theorem}
While the setting of Theorem \ref{thm:main} resembles
the planar case of \cite{MR2346510}, our proof builds mainly on \cite{ameur2011fluctuations} and is presented in Section \ref{sec_a}. 

Theorem \ref{thm:main} is applicable for example to general \emph{polyanalytic Ginibre ensembles} \cite{haimi2013polyanalytic, haimi2014bulk}, where $K$ is the reproducing kernel of the space \eqref{eq_m} and $Q: \mathbb{C} \to [0,\infty)$ is an adequate \emph{potential function}. Under mild assumptions on $Q$, a suitable adaptation of
the estimates of \cite{haimi2013polyanalytic, haimi2014bulk}, which concern finite-particle systems, shows that $K$ satisfies the assumptions of Theorem \ref{thm:main}. (While we are not aware of a citable reference for this fact, we expect to provide the particulars in a forthcoming article.)

Concerning translation invariant processes, we obtain the following application of Theorem \ref{thm:main}.
\begin{cor}\label{cor:main}
Let $K: \CC \times \CC \to \CC$ be a Hermitian symmetric reproducing kernel acting on $L^2(\CC, dA)$. Suppose that 
\begin{align}\label{A}
|K(z,w)|=\phi(z-w), \qquad z,w\in\CC,
\end{align}
for a bounded continuous function $\phi:\CC \to [0,\infty)$ such that
\begin{align}\label{B}
\int_\CC (1+|z|^3) \phi(z) \, dA(z) <\infty.
\end{align}
Define the dilated kernels
\begin{align}\label{eq_dilated}
K_\rho(z,w)= \rho K(\sqrt{\rho}z, \sqrt{\rho}w), \qquad \rho>0.
\end{align}

For a compactly supported, real-valued $C^3$ test function $f$ set
\begin{align}
\label{eq_mf}
\mu_f &:= \phi(0) \cdot \int_{\CC} f(z) \,dA(z),
\\
\sigma^2_f &:= \frac{1}{2} \int_\CC \int_\CC (\nabla f(z) \cdot w )^2 \phi(w)^2 \, dA(z) dA(w).
\end{align} 
Then
$\tr_\rho(f)-\rho \cdot \mu_f \to N(0,\sigma_f^2)$ in distribution.
\end{cor}
Note that if $\phi(0) \not= 0$ Corollary \ref{cor:main} implies that \eqref{eq:variance_cond} is not satisfied. A proof of Corollary \ref{cor:main} is presented in Section \ref{sec_b}. As an example, the Weyl-Heisenberg kernel \eqref{eq_wh} satisfies the hypotheses for any Schwartz function $g \in \mathcal{S}(\mathbb{R})$, and a suitable envelope is provided by
\begin{align*}
\phi(z) = \Big|
\int_{\mathbb{R}} {g(t)}\overline{g(t-x)} e^{-2\pi i t y} \,dt \Big|, \qquad z=x + i y \in \mathbb{C}.
\end{align*}
Concerning possible generalizations of Corollary \ref{cor:main}, we remark that an analogous setting in higher dimension leads to smooth observables with expectation and variance growing polynomially on $\rho$ (see Remark \ref{rem_dil}). Similarly, as we show in Proposition \ref{prop_non},
if we relax the assumptions of Corollary \ref{cor:main} to allow for non-reproducing kernels, then for each such non-reproducing kernel the expectation and variance of $\tr_\rho(f)$ also grow polynomially in $\rho$. Thus, in these cases, \eqref{eq:variance_cond} is satisfied for some $\delta>0$ and Soshnikov's asymptotic normality theorem can be invoked.

Before proceeding to the proofs, we make some final comments. 
An abstract setting in which linear statistics can be shown to be asymptotically normal was recently introduced in \cite{dinh2021quantitative}, where, moreover, quantitative convergence rates are derived. In the planar case, however, \cite[Model 1.2]{dinh2021quantitative} is not compatible with the envelope assumptions in Theorem \ref{thm:main}, and, as in Soshnikov's result, $\var[\tr_\rho(f)]$ needs to diverge 
as $\rho \to \infty$ \cite[Proposition 4.9]{dinh2021quantitative}.

Second, we remark that Theorem \ref{thm:main} can probably be refined to cover kernels which satisfy the assumptions only near the support of the test function, provided that \eqref{eq_rf} is replaced by a suitably local approximate formula
(while in analytic settings, reproducing formulas are exactly local). Such an extension would allow us to also treat finite particle systems in regimes where points congregate in a compact set (droplet) as long as the test function is supported away from its boundary. However, the expected gain in generality did not seem to merit the inclusion of additional technicalities. A more promising future direction is an extension of Theorem \ref{thm:main} to allow for droplets and test functions that interact with them, paralleling what is known in the model cases \cite{rider2007noise, ameur2015random, haimi2016central}.

The reminder of the article is organized as follows: Section \ref{sec_not} clarifies the notation, Theorem \ref{thm:main} is proved in Section \ref{sec_a}, while Section \ref{sec_b}
contains a proof of Corollary \ref{cor:main} and a discussion on more general settings and the applicability of Soshnikov's theorem.

\section{Notation}\label{sec_not}
The real and imaginary parts of $z \in \mathbb{C}$ are denoted $\mathrm{Re}(z)$ and $\mathrm{Im}(z)$. We shall also identify $\CC\simeq\mathbb{R}^2$
and write $z=(z^{(1)},z^{(2)})$. Derivatives with respect to real variables are denoted as follows: 
for $F:\CC^n \to \mathbb{C}$, we let $\der_{rj} F(z_1, \ldots, z_n) \in \CC$ be the partial derivative of $F$ with respect to $z^{(j)}_r$, $1 \leq r \leq n$, $j=1,2$. For $z_0 \in \CC$, we denote $\vec{z_0} = (z_0,\ldots,z_0) \in \CC^k$ if the vector length $k$ is clear from the context.

The differential of the Lebesgue (area) measure on $\CC$ is denoted $dA$. We also let $dA$ denote the corresponding product measure on $\CC^n$. For example, depending on convenience and readability we write interchangeably
\begin{align*}
\int_{\CC^n} F \,dA &= 
\int_{\CC^n} F(z_1,\ldots,z_n) \,dA(z_1,\ldots,z_n)
=
\int_{\CC^n} F(z_1,\ldots,z_n) \,dA(z_1)\ldots dA(z_n)
\\
&=
\int_{z_1, \ldots, z_n \in \CC} F(z_1,\ldots,z_n) \,dA.
\end{align*}
For two non-negative functions $F,G: X \to [0,\infty)$ we write $F \lesssim G$ if there exists a constant $C>0$ such that $F(x) \leq C G(x)$, for all $x \in X$.
We also use the following asymptotic notation with respect to the parameter $\rho$:
$F(\rho) = O(\rho^\alpha)$ means that
\begin{align}\label{eq_1}
\limsup_{\rho \to \infty} \rho^{-\alpha} F(\rho) <\infty
\end{align}
 while
$F(\rho) = o(\rho^\alpha)$ means that
\begin{align}\label{eq_2}
\lim_{\rho \to \infty} \rho^{-\alpha} F(\rho) =0.
\end{align}
For example, \eqref{A3} means that
$\int_\CC |z|^3 \phi_\rho(z) dA(z) = o(1/\rho)$. When the function $F$ depends on additional parameters besides $\rho$, the limits \eqref{eq_1} and \eqref{eq_2} may or may not hold uniformily on these. Such dependencies are clarified in each case.
\section{Asymptotic normality}\label{sec_a}
The overall strategy is as follows. To show that smooth linear statistics are asymptotically normal we apply the cumulant method (see, e.g., \cite[Appendix A.3]{nourdin2012normal}). We use the exact expressions for cumulants derived in \cite{soshnikov2002gaussian} to then obtain asymptotic expansions by Taylor expanding the test function to second order, as in \cite{rider2007noise,ameur2011fluctuations}. Elaborating on \cite{ameur2011fluctuations}, we show that the reproducing property of the correlation kernel leads to certain significant cancellations, which allow us to conclude that higher order cumulants vanish asymptotically.

\subsection{Cumulants}
Let us fix a family of reproducing kernels $K_\rho$ satisfying (A1), (A2) and (A3) and a compactly supported $C^3$ function $f: \mathbb{C} \to \mathbb{R}$. Without loss of generality, we can assume that the kernel envelops are symmetric
\begin{align*}
\phi_\rho(z)=\phi_\rho(-z), \qquad z\in\CC, \rho>0.
\end{align*}
The cyclic products of the correlation kernel $K_{\rho}$ are denoted
\begin{align*}
R_{\rho, k}(z_1, \ldots, z_k):= K_{\rho}(z_1, z_2) \cdots K_{\rho}(z_{k-1}, z_k) K_{\rho}(z_k, z_1), \qquad
z_1,\ldots,z_k \in \mathbb{C}.
\end{align*}
These are instrumental to compute the cumulants of the linear statistics $\tr_\rho(f)$. Indeed, let us associate with $f$ the expression
\begin{align}\label{eq_defG}
G_k(z_1,\ldots z_k)=\sum_{j=1}^k\frac{(-1)^{j-1}}{j}\sum_{\substack{k_1+k_2+\ldots +k_j=k \\ k_1,\ldots, k_j\geq 1}}\frac{k!}{k_1!k_2!\cdots k_j!}\prod_{l=1}^jf(z_l)^{k_l},\qquad z_1,\ldots,z_k \in \mathbb{C}.
\end{align}
As shown in \cite[Lemma 1]{soshnikov2002gaussian},
the cumulants $C_{\rho,k}(f)$ of $\tr_\rho(f)$ are
\begin{align}\label{eq:cum_integral_1}
C_{\rho,k}(f) = \int_{\mathbb{C}^{k}} G_k(z_1, \ldots, z_k) R_{\rho, k}(z_1, \ldots, z_k) \,dA(z_1) \cdots dA(z_k).
\end{align}
(This is a simplification of a more general formula in \cite[Lemma 1]{soshnikov2002gaussian} that concerns possibly non-reproducing kernels.)
By the reproducing property of the kernel $K_\rho$, we can extend integration in \eqref{eq:cum_integral_1} to a new variable $z_0$ and write
\begin{align} \label{eq:cum_integral_2}
C_{\rho,k}(f)=\int_{\mathbb{C}^{k+1}} G_k(z_1, \ldots, z_k) R_{\rho, k+1}(z_0,z_1, \ldots, z_k) \,dA(z_0) dA(z_1) \cdots dA(z_k).
\end{align}
This observation is a main tool in \cite{ameur2011fluctuations} and it is similarly fundamental for our proof.

\subsection{Taylor expansions}\label{sec_taylor}
As in \cite{rider2007noise,ameur2011fluctuations} we start by expanding $G_k$ up to second order. We write
\begin{equation} \label{eq:taylor}
G_k(z_1, \ldots, z_k)= G_k(\vec{z_0}) + \Ti(z_0, \ldots, z_k)+ \Tii(z_0, \ldots, z_k)+ \rk(z_0, \ldots, z_k)
\end{equation}
where 
\begin{align}
\vec{z_0} &= (z_0, \ldots, z_0),
\\
\label{eq_b1}
\Ti(z_0, \ldots, z_k)&= \sum_{1\leq r \leq k, 1\leq j \leq 2} \der_{rj}G_k(\vec{z_0})\big(z_r^{(j)}- z_0^{(j)}\big),
\\
\label{eq_b2}
\Tii(z_0, \ldots, z_k)&=  \frac12 \sum_{1\leq r_1, r_2 \leq k, 1\leq j_1, j_2 \leq 2} \der_{r_1j_1} \der_{r_2j_2}G_k(\vec{z_0}) \big(z_{r_1}^{(j_1)}-z_0^{(j_1)}\big) \big(z_{r_2}^{(j_2)}- z_0^{(j_2)}\big)
\end{align}
and the reminder term $\rk(z_0, \ldots, z_k)$ satisfies
\begin{align}\label{eq_taylor_error}
|\rk(z_0, \ldots, z_k)| \lesssim 
\sum_{j=1}^k |z_j-z_0|^3,
\end{align}
with an implied constant that depends on $f$ and $k$, but is uniform in the center of the expansion $z_0$ because $f$ (and hence $G_k$) have globally bounded third order derivatives.

We now inspect the effect of the Taylor approximation in the cumulant expression \eqref{eq:cum_integral_2}.

\begin{lemma} \label{lem:local}
	Let $B \subset \mathbb{C}$ be a compact set such that
	$\mathrm{dist} (\supp(f), B^c)>0$. Then, for $k\geq 2$,
	\begin{multline}\label{eq_aaa}
	C_{\rho, k}(f)  
	= \int_{B \times \mathbb{C}^{k}} \bigg( G_k(\vec{z_0})+ \Ti(z_0, \lsdots, z_k) + \Tii(z_0, \lsdots, z_k) \bigg) \,\times 
	\\
	\,R_{\rho, k+1}(z_0, \lsdots, z_k) \,dA(z_0,\lsdots,z_k) + \mathrm{o}(1),
	\end{multline}
	where the convergence in the implied limit depends on $f$, $B$, and $k$.
\end{lemma}
\begin{proof}
\noindent {\bf Step 1}. Let $\varepsilon := \mathrm{dist} (\supp(f), B^c)>0$, and fix $k \geq 2$. We first show that
\begin{align}\label{eq_a}
C_{\rho, k}(f) = \int_{B \times \mathbb{C}^k} G_k(z_1, \lsdots, z_k) R_{\rho, k+1}(z_0, \lsdots, z_k) \,dA(z_0, \lsdots, z_k) + \mathrm{o}(1). 
\end{align}
Inspecting \eqref{eq_defG} we see that $G_k$ is a sum of terms of the form $f_1(z_1) \cdots f_k(z_k)$, where each $f_j$ is a power of $f$, and thus compactly supported or is identically $1$, 
and not all $f_j$ are $1$. Hence, it suffices to show that for each $k \in \mathbb{N}$,
\begin{align} \label{eq:local1}
\begin{aligned}
	&\int_{\mathbb{C}^{k+1}} f_1(z_1) \cdots f_k(z_k) R_{\rho, k+1}(z_0, \lsdots, z_k) \,dA(z_0,\lsdots, z_k) \\
	&\quad= \int_{B \times \mathbb{C}^k } f_1(z_1) \cdots f_k(z_k) R_{\rho, k+1}(z_0, \lsdots, z_k) \,dA(z_0,\lsdots, z_k) + \mathrm{o}(1).
\end{aligned}
\end{align}
Moreover, we can assume that $f_1$ is compactly supported; otherwise, it would be identically $1$ and we would eliminate the variable $z_1$ using the reproducing property of the kernel $K_\rho$, and relabel the rest of the points to obtain another instance of \eqref{eq:local1} with a smaller value of $k$. (Recall that not all $f_j$ are $1$.)

Let us now estimate the error term corresponding to \eqref{eq:local1}.
We exploit the support of $f_1$ and bound the cyclic product $R_{\rho, k+1}(z_0,z_1, \ldots, z_k)$ applying \eqref{A1} to the factor $K_\rho(z_1,z_2)$ and \eqref{A2} to the others:
\begin{align*}
	&E_k:=\bigg| \int_{B^c \times \CC^k} f_1(z_1) \csdots f_k(z_k) R_{\rho, k+1}(z_0, \lsdots, z_k) \,dA(z_0, \lsdots,z_k) \bigg|
	\\
	&\,\lesssim \rho  \int_{z_0 \in B^c} \int_{z_1 \in B,\,|z_1-z_0|\geq \varepsilon} \int_{z_2, \ldots, z_k \in \CC} \phi_\rho(z_1-z_0) \phi_\rho(z_3-z_2) \csdots \phi_\rho(z_k-z_{k-1}) \phi_\rho(z_0-z_k) \, dA
 	\\
	&\,\lesssim \rho \int_{z_0 \in \CC} \int_{z_1 \in B} \int_{z_2, \ldots, z_k \in \CC} |z_1-z_0|^3 \phi_\rho(z_1-z_0) \phi_\rho(z_3-z_2) \csdots \phi_\rho(z_k-z_{k-1}) \phi_\rho(z_0-z_k) \,dA
	\\
	&\,=\rho \int_{z_1 \in B} \int_{z_0 \in \CC} |z_1-z_0|^3 \phi_\rho(z_1-z_0) \int_{z_2, \lssdots, z_k \in \CC} \phi_\rho(z_3-z_2) \csdots \phi_\rho(z_k-z_{k-1}) \phi_\rho(z_0-z_k) \,dA\,dA(z_0,z_1)
\end{align*}
The integral on $z_2,\ldots,z_k$ is seen to be $O(1)$ by integrating in the order $z_2, \ldots, z_k$ and applying \eqref{A2}. The remaining expression can be bounded by means of \eqref{A3}:
\begin{align*}
	E_k \lesssim \rho \int_{z_1 \in B} \int_{z_0 \in \CC} |z_1-z_0|^3 \phi_\rho(z_1-z_0) \,dA(z_0)\,dA(z_1) \lesssim \rho \cdot \mathrm{o}(1/\rho) = \mathrm{o}(1).
\end{align*}

\medskip

\noindent {\bf Step 2}. By the Taylor expansion \eqref{eq:taylor}
and \eqref{eq_taylor_error}, it suffices to prove that
for all $j=1,\ldots,k$,
$$
\int_{B \times \CC^k} |z_j-z_0|^3 R_{\rho, k+1}(z_0, \lsdots, z_k) \,dA(z_0,\lsdots,z_k) = \mathrm{o}(1).
$$
We use the bound $|z_j-z_0|^3 \lesssim \sum_{n=1}^j |z_{n}-z_{n-1}|^3$. We suppose first that $j<k$ and estimate $|K_\rho(z_j,z_{j+1})| \lesssim \rho$ by \eqref{A1}, and the other factors of $R_{\rho, k+1}(z_0, \ldots, z_k)$ by their envelopes to obtain
\begin{align*}
E'_{j} &:= \bigg|\int_{B \times \CC^k} |z_j-z_{j-1}|^3 R_{\rho, k+1}(z_0, \lsdots, z_k) \,dA(z_0, \lsdots,z_k)\bigg|
\\
&\lesssim \rho
\int_{z_0 \in B} \int_{z_j \in \mathbb{C}} |z_j-z_{j-1}|^3 \phi_\rho(z_j-z_{j-1}) 
\int \phi_\rho(z_0-z_k)
\prod_{\stackrel{n=1,\ldots,k}{n\not=j,j+1}} \phi_\rho(z_n-z_{n-1}) \,dA \, dA(z_j) \,dA(z_0),
\end{align*}
where the innermost integral runs over all $z_n \in \mathbb{C}$ with $1 \leq n \leq k$ and $n\not=j$. 
By \eqref{A2}, if we integrate in the order $z_{j+1},\ldots,z_k$ we obtain
\begin{align*}
\int_{z_k} \cdots \int_{z_{j+1}} \phi_\rho(z_0-z_k) \prod_{\stackrel{n=1,\ldots,k}{n\not=j,j+1}} \phi_\rho(z_n-z_{n-1}) \,dA
\lesssim \prod_{n=1}^{j-1} \phi_\rho(z_n-z_{n-1}).
\end{align*}
Second, integration on $z_j$ brings the factor
\begin{align*}
\rho \int_{z_j \in \mathbb{C}} |z_j-z_{j-1}|^3 \phi_\rho(z_j-z_{j-1}) 	\, dA(z_j)
\lesssim \rho \cdot \mathrm{o}(1/\rho) = \mathrm{o}(1). 
\end{align*}
Finally, by \eqref{A2}, integration in the order $z_{j-1},\ldots,z_{1}$ yields
\begin{align*}
E'_{j} \lesssim \mathrm{o}(1) \cdot \int_{z_0\in B}
\int_{z_1} \cdots \int_{z_{j-1}}  \prod_{n=1}^{j-1} \phi_\rho(z_n-z_{n-1}) \, dA
\lesssim \mathrm{o}(1) \cdot |B| = \mathrm{o}(1).
\end{align*}
The case $j=k$ is similar: we bound $|K_\rho(z_{0},z_{k})| \lesssim \rho$ and proceed as before by estimating integrals in the order $z_k, \ldots, z_0$.
\end{proof}
\subsection{Symmetry properties of cumulants}
We now collect various symmetry properties of the functions $G_k$ in the following lemma, which elaborates on \cite[Lemmas 3.1, 3.2, 3.3]{ameur2011fluctuations} by recording certain additional symmetries implicit in the proof of those results.

\begin{lemma} \label{lem:G_k}
	The following holds for all $z_0 \in \CC$ and
	$k \geq 3$:
	\begin{enumerate} [(i)]
		\item[(i)] $G_k(\vec{z_0})=0$, 
		\item[(ii)] $\sum_{1\leq r \leq k} (\der_{rj}G_k)(\vec{z_0}) =0$ for all $1 \leq j \leq 2$,
		\item[(iii)]  $\sum_{1 \leq r_1, r_2 \leq k}(\der_{r_1m} \der_{r_2l}G_k)(\vec{z_0}) =0$ for all $1 \leq m, l \leq 2$,
		\item[(iv)] $\sum_{1 \leq r \leq k} (\der_{rm}\der_{rl}G_k)(\vec{z_0})=0$ for all $1 \leq l,m \leq 2$,
		\item[(v)] $\sum_{1 \leq r_1 \neq r_2 \leq k}(\der_{r_1m} \der_{r_2l}G_k)(\vec{z_0}) =0$, for all $1 \leq m, l \leq 2$.
	\end{enumerate}
\end{lemma}
\begin{proof}
The proof of (i) is exactly as in \cite[Lemma 3.1]{ameur2011fluctuations}. Parts (ii) and (iii) follow by taking derivatives on (i). We now turn to (iv), that is a refined version of \cite[Lemma 3.3]{ameur2011fluctuations} which shows that $\Delta_{\mathbb{R}^{2k}} G_k(\vec{z_0})=0$.
	
	Let $k_1, \ldots, k_j$ be positive integers such that $k_1 + \cdots+ k_j=k$. Then for $1 \leq r \leq j$, $1 \leq m,l \leq 2$ and  $z_1, \ldots, z_k \in \CC$ we have
	\begin{align}
	\begin{aligned}
	&\der_{rm}\der_{rl} \bigg( \prod_{n=1}^j f(z_n)^{k_n} \bigg) \\
	&\quad= \prod_{1 \leq n \leq j, n \neq r} f(z_n)^{k_n} \bigg( k_r(k_r-1) f(z_r)^{k_r-2} \der_{m}f(z_r) \der_{l} f(z_r)+ k_r f(z_r)^{k_r-1} \der_{m} \der_{l}f(z_r) \bigg).
	\end{aligned}
	\end{align}
	This gives 
	\begin{align*}
	\sum_{1 \leq r \leq k} (\der_{rm}\der_{rl}G_k)(\vec{z_0}) = f(z_0)^{k-2} \der_mf(z_0) \der_lf(z_0) \cdot S_k + f(z_0)^{k-1} \der_m\der_lf(z_0) \cdot S'_k,
	\end{align*}
	where
	\begin{align*}
	S_k &= 	\sum_{j=1}^k \frac{(-1)^{j-1}}{j} \sum_{\substack{k_1+k_2+\ldots +k_j=k \\ k_1,\ldots, k_j\geq 1}} \frac{ k!(k_1(k_1-1)+ \cdots+ k_j(k_j-1))}{k_1! \cdots k_j!},
	\\
	S'_k &= \sum_{j=1}^k \frac{(-1)^{j-1}}{j}
	\sum_{\substack{k_1+k_2+\ldots +k_j=k \\ k_1,\ldots, k_j\geq 1}} \frac{k \cdot k!}{k_1! \cdots k_j!}.
	\end{align*}
	$S_k$ is shown to be zero in the proof of \cite[Lemma 3.3]{ameur2011fluctuations} (see Eqs. (3.5) and (3.9)
	in \cite{ameur2011fluctuations}), while $S'_k$ shown to be zero in \cite[Lemma 3.1 and Eq. (3.1)]{ameur2011fluctuations}. 
	
	Finally, (v) follows by combining (iii) and (iv).	
\end{proof}

\subsection{Proof of Theorem \ref{thm:main}}
\mbox{}

\noindent {\bf Step 1}. We invoke Lemma \ref{lem:local} with an adequate choice of $B$ and use the representation \eqref{eq_aaa}. First, note that, by Lemma \ref{lem:G_k}, 
\begin{align}\label{eq_c0}
G_k(\vec{z_0})=0.
\end{align}
Let us show that
\begin{align}\label{eq_c}
\int_{B\times\mathbb{C}^k} \Ti(z_0,\lsdots,z_k) R_{\rho,k+1}(z_0,\lsdots,z_k) \,dA(z_0,\lsdots,z_k) = 0.
\end{align}
We inspect \eqref{eq_b1}, fix $j \in \{1,2\}$ and use the reproducing property of the kernel $K_\rho$ to compute
\begin{align*}
	&\sum_{r=1}^k \int_{B \times \mathbb{C}^{k}} \der_{rj}G_k(\vec{z_0}) \big(z_r^{(j)}-z_0^{(j)}\big) R_{\rho,k+1}(z_0, \ldots, z_k) \,dA(z_0)\cdots dA(z_k)
	\\
	&\quad= \sum_{r=1}^k \int_{B} \int_{\mathbb{C}} \der_{rj}G_k(\vec{z_0}) \big(z_r^{(j)}-z_0^{(j)}\big) R_{\rho,2}(z_0, z_r) \,dA(z_r) dA(z_0)
	\\
	&\quad=  \int_{B \times \mathbb{C}}  \sum_{r=1}^k  \der_{rj}G_k(\vec{z_0})  \big(z^{(j)}-z_0^{(j)}\big) R_{\rho,2}(z_0, z) \,dA(z) dA(z_0).
\end{align*}
The sum on $r$ vanishes by Lemma \ref{lem:G_k}, which proves \eqref{eq_c}.

\medskip
\noindent {\bf Step 2}. We turn to the second order terms $\Tii$ given by \eqref{eq_b2}. We separate the terms for which $r_1=r_2$
and classify the rest according to $(j_1,j_2)$:
\begin{align}\label{eq_f}
\int_{B \times \mathbb{C}^{k}} \Tii(z_0, \lsdots, z_k) R_{\rho,k+1}(z_0, \lsdots, z_k) \, dA(z_0,\lsdots, z_k)=
\frac12\big[
A_0+A_{1,1} +A_{1,2}+A_{2,1}+A_{2,2}\big],
\end{align}
where
\begin{align*}
A_0 &= \int_{B \times \mathbb{C}^{k}} \sum_{\substack{1\leq r \leq k\\1\leq j_1, j_2 \leq 2}} \der_{r j_1} \der_{r j_2}G_k(\vec{z_0}) \big(z_{r}^{(j_1)}-z_0^{(j_1)}\big)\big(z_{r}^{(j_2)}-z_0^{(j_2)}\big) R_{\rho, k+1}(z_0, \lsdots, z_k) \,dA(z_0,\lsdots,z_k),
\\
A_{j_1,j_2} &= \int_{B \times \mathbb{C}^{k}} \sum_{\substack{1\leq r_1, r_2 \leq k\\r_1 \neq r_2}} \der_{r_1j_1} \der_{r_2j_2}G_k(\vec{z_0}) \big(z_{r_1}^{(j_1)}-z_0^{(j_1)}\big)\big(z_{r_2}^{(j_2)}-z_0^{(j_2)}\big) R_{\rho, k+1}(z_0, \lsdots, z_k) \,dA(z_0,\lsdots,z_k).
\end{align*}

\medskip

\noindent {\bf Step 3}. We use the reproducing formula and Lemma \ref{lem:G_k} to compute
\begin{align*}
	A_0	&= \sum_{\substack{1\leq r \leq k\\1\leq j_1, j_2 \leq 2}} 	\int_{B} \int_{\mathbb{C}} 
	\der_{rj_1} \der_{rj_2}G_k(\vec{z_0}) 
\big(z_r^{(j_1)}-z_0^{(j_1)}\big) \big(z_r^{(j_2)}-z_0^{(j_2)}\big)  R_{\rho, 2}(z_0, z_r)\, dA(z_r) dA(z_0)
	\\
	&= \sum_{1\leq j_1, j_2 \leq 2}
	 \int_{B} \int_{ \mathbb{C}} 
	\sum_{{1\leq r \leq k}}
			\der_{rj_1} \der_{rj_2}G_k(\vec{z_0}) 
	 \big(z^{(j_1)}-z_0^{(j_1)}\big) \big(z^{(j_2)}-z_0^{(j_2)}\big)  R_{\rho, 2}(z, z_0)\,dA(z)dA(z_0)
	=0.
\end{align*}

\medskip

\noindent {\bf Step 4}. Similarly, for $j\in\{1,2\}$,
\begin{align}\label{eq_d1}
A_{jj}&=\sum_{\substack{1\leq r_1, r_2 \leq k\\r_1 < r_2}} \int_{B}  \der_{r_1 j} \der_{r_2j}G_k(\vec{z_0}) \int_{\mathbb{C}^2}
\big(z_{r_1}^{(j)}-z_0^{(j)}\big)\big(z_{r_2}^{(j)}-z_0^{(j)}\big)
R_{\rho, 3}(z_0, z_{r_1}, z_{r_2}) \,dA(z_{r_1}, z_{r_2})dA(z_0)\,+
\\\label{eq_d2}
&\,
\sum_{\substack{1\leq r_1, r_2 \leq k\\r_1 > r_2}}  \int_{B}  \der_{r_1 j} \der_{r_2j}G_k(\vec{z_0}) 
\int_{\mathbb{C}^2} \big(z_{r_1}^{(j)}-z_0^{(j)}\big)\big(z_{r_2}^{(j)}-z_0^{(j)}\big)
R_{\rho, 3}(z_0, z_{r_2}, z_{r_1}) \,dA(z_{r_1},z_{r_2})dA(z_0).
\end{align}
The inner integrals in \eqref{eq_d1} and \eqref{eq_d2} are equal and independent of $r_1$ and $r_2$. Thus, denoting by $I$ their common value, Lemma \ref{lem:G_k} gives
\begin{align*}
A_{jj}= I \cdot \sum_{\substack{1\leq r_1, r_2 \leq k\\r_1 \neq r_2}} \der_{r_1 j} \der_{r_2j}G_k(\vec{z_0}) = 0.
\end{align*}

\medskip

\noindent {\bf Step 5}. For the cross terms $A_{21}=A_{12}$ we invoke again the reproducing formula,
\begin{align}\label{eq_e1}
A_{12} &= \sum_{\substack{1\leq r_1, r_2 \leq k\\r_1 < r_2}}  \int_{B} \der_{r_1 1} \der_{r_2 2}G_k(\vec{z_0})  \int_{\mathbb{C}^{2}} 
\big(z_{r_1}^{(1)}-z_0^{(1)}\big)\big(z_{r_2}^{(2)}-z_0^{(2)}\big)
R_{\rho, 3}(z_0, z_{r_1}, z_{r_2}) \,dA(z_{r_1},z_{r_2})dA(z_0) \,+
\\
\label{eq_e2}
&\,\sum_{\substack{1\leq r_1, r_2 \leq k\\r_1 > r_2}} \int_{B}
 \der_{r_1 1}  \der_{r_2 2}G_k(\vec{z_0})  \int_{\mathbb{C}^{2}} 
\big(z_{r_1}^{(1)}-z_0^{(1)}\big)\big(z_{r_2}^{(2)}-z_0^{(2)}\big)
R_{\rho, 3}(z_0, z_{r_2}, z_{r_1}) \,dA(z_{r_1},z_{r_2})dA(z_0),
\end{align}
and note that, by the Hermitian symmetry of the kernel $K_\rho$, the inner integral in \eqref{eq_e1} --- denoted by $I'$ --- is the complex conjugate of the inner integral in \eqref{eq_e2}. Thus,
\begin{align}\label{eq_nn}
\mathrm{Re}(A_{21})=\mathrm{Re}(A_{12})=
\mathrm{Re}(I') \cdot \sum_{\substack{1\leq r_1, r_2 \leq k\\r_1 \neq r_2}} \der_{r_1 1} \der_{r_2 2}G_k(\vec{z_0}) = 0.
\end{align}

We combine the previous calculations with \eqref{eq_f} to conclude
\begin{align}\label{eq_c2}
\mathrm{Re} \Big[ \int_{B\times\mathbb{C}^k} \Tii(z_0,\ldots,z_k) R_{\rho,k+1}(z_0,\ldots,z_k) \,dA(z_0)\ldots dA(z_k) \Big] = 0.
\end{align}

\noindent {\bf Step 6}. Since $f$ is real-valued, so is the cumulant \eqref{eq:cum_integral_2}. Hence, \eqref{eq_aaa} together with \eqref{eq_c0}, \eqref{eq_c}, \eqref{eq_nn}, and \eqref{eq_c2} yield
\begin{align*}
C_{\rho,k} = \mathrm{Re} (C_{\rho,k}) = o(1),
\end{align*}
as claimed.

Towards proving (i), suppose that $\mathrm{Var}[\tr_\rho(f)] \longrightarrow \sigma^2>0$. To conclude the desired asymptotic normality, the cumulant method requires the uniform integrability of the random variables $\tr_\rho(f)- \mathbb{E}[\tr_\rho(f)]$
(see, e.g., \cite[Appendix A.3]{nourdin2012normal}). This is the case because their variances are uniformly bounded. In the degenerate case $\sigma=0$, $\tr_\rho(f)- \mathbb{E}[\tr_\rho(f)]\to 0$ in probability, and thus in distribution.

Similarly, suppose that
$\liminf_{\rho\to\infty} \var [\tr_\rho(f)] > 0$. It is easy to see that, due to the enveloping conditions,  $\var [\tr_\rho(f)]=O(1)$; see, for example, the proof of Corollary \ref{cor:main} below. Let $\tilde{f}_\rho := \var [\tr_\rho(f)]^{-1} f$. Then, inspecting \eqref{eq_defG}, we see that the cumulants $\widetilde{C}_{\rho,k}$ of $\tr_\rho(\tilde{f}_\rho)$ satisfy, for $k \geq 3$,
\begin{align*}
\widetilde{C}_{\rho,k} = \var [\tr_\rho(f)]^{-k} \cdot  C_{\rho,k}(f) = o(1),
\end{align*}
and we conclude by the cumulant method that $\tr_\rho(\tilde{f}_\rho)$ is asymptotically normal. This proves (ii).
\qed

\section{The translation invariant case}\label{sec_b}
We now specialize Theorem \ref{thm:main} to dilated kernels with translation invariant correlations.
\begin{proof}[Proof of Corollary \ref{cor:main}] 
Let $K$ satisfy the assumptions of Corollary \ref{cor:main}. Without loss of generality, we assume that $\phi$ is symmetric. The expectation of $\tr_\rho(f)$ is given by \eqref{eq_e} and reduces to
$\rho \mu_f$, cf. \eqref{eq_mf}.
Since $K$ is reproducing,
\begin{align*}
\var[\tr_\rho(f)] = \frac12 \int_\CC \int_\CC \big[f(z)-f(w)\big]^2 |K_\rho(z,w)|^2 \, dA(w) dA(z),
\end{align*}
as can be easily checked by inspecting \eqref{eq:cum_integral_1} with $k=2$.
Let $B \subset \CC$ be a compact set such that $\varepsilon:=\mathrm{dist} (\supp(f), B^c)>0$. We split the variance as
\begin{align}\label{eq_a1}
\begin{aligned}
\var[\tr_\rho(f)] &= \frac12 \int_B \int_\CC \big[f(z)-f(w)\big]^2 |K_\rho(z,w)|^2 \, dA(w) dA(z)
\\
&\qquad+ \frac12 \int_{B^c} \int_\CC f(w)^2 |K_\rho(z,w)|^2 \, dA(w) dA(z).
\end{aligned}
\end{align}
For the second term, note that when $w \in \supp(f)$ and $z \in B^c$,
$|z-w| \geq \varepsilon$. Thus, using \eqref{B}, and $\phi^2 \leq \|\phi\|_\infty \cdot \phi$,
\begin{align}\label{eq_a2}
\begin{aligned}
&\int_{B^c} \int_\CC f(w)^2 |K_\rho(z,w)|^2 \, dA(w) dA(z)
\\
&\qquad\leq  \varepsilon^{-3} \| f \|^2_\infty
\int_{B^c} \int_{\supp(f)} |z-w|^3 |K_\rho(z,w)|^2 \, dA(w) dA(z)
\\
&\qquad \lesssim \rho^2 \int_{B^c} \int_{\supp(f)} |z-w|^3 \phi(\sqrt{\rho}(z-w)) \, dA(w) dA(z)
\\
&\qquad =  \rho^{-1/2} \int_{\supp(f)} \rho \int_{B^c} |\sqrt{\rho}(z-w)|^3 \phi(\sqrt{\rho}(z-w)) \, dA(z) dA(w)
\lesssim \rho^{-1/2}.
\end{aligned}
\end{align}
Taylor expanding $f$ to order 1 we see that
\begin{align*}
[f(w)-f(z)]^2=[\nabla f(z) \cdot (w-z)]^2 + E(z,w),
\end{align*}
where $|E(z,w)| \leq C |z-w|^3$ since $f \in C^3$ and compactly supported. Thus
\begin{align}\label{eq_a3}
\begin{aligned}
&\int_B \int_\CC |E(z,w)| |K_\rho(z,w)|^2 \, dA(w) dA(z) \\
&\qquad\lesssim
\rho^2 \int_B \int_\CC |z-w|^3 \phi(\sqrt{\rho}(z-w)) \, dA(w) dA(z) = O(\rho^{-1/2}).
\end{aligned}
\end{align}
Combining \eqref{eq_a1}, \eqref{eq_a2} and \eqref{eq_a3} we obtain
\begin{align*}
\var[\tr_\rho(f)] &= \frac12 \int_B \int_\CC \big[\nabla f(z) \cdot (w-z)\big]^2 |K_\rho(z,w)|^2 \, dA(w) dA(z) + O(\rho^{-1/2})
\\
&=\frac12 \rho^2 \int_B \int_\CC \big[\nabla f(z) \cdot (w-z)\big]^2 \phi(\sqrt{\rho}(z-w))^2 \, dA(w) dA(z) + O(\rho^{-1/2})
\\
&=\frac12 \rho^2 \int_B \int_\CC \big[\nabla f(z) \cdot u\big]^2 \phi(\sqrt{\rho}u)^2 \, dA(u) dA(z) + O(\rho^{-1/2})
\\
&=\frac12 \int_B \rho \int_\CC \big[\nabla f(z) \cdot (\sqrt{\rho} u)\big]^2 \phi(\sqrt{\rho}u)^2 \, dA(u) dA(z) + O(\rho^{-1/2})
\\
&=\frac12 \int_\CC \int_\CC \big[\nabla f(z) \cdot u\big]^2 \phi(u)^2 \, dA(u) dA(z) + O(\rho^{-1/2}).
\end{align*}
We now invoke Theorem \ref{thm:main} (part (i)). Suitable envelopes are provided by $\phi_\rho(z) = \rho \phi(\sqrt{\rho}z)$.
\end{proof}
\begin{rem}\label{rem_dil}
If we consider a reproducing kernel $K$ acting on $L^2(\mathbb{R}^d)$, satisfying $|K(z,w)|=\phi(z-w)$, and dilate it by $K_\rho(z,w) = \rho^{d/2} K\big(\sqrt{\rho}(z-w)\big)$, then under similar assumptions as in Corollary \ref{cor:main} reinspection of the proof of Corollary \ref{cor:main} shows
\begin{align}\label{eq_lll}
\var[\tr_\rho(f)] = \frac12 \rho^{d/2-1} \int_{\mathbb{R}^d} \int_{\mathbb{R}^d} \big[\nabla f(z) \cdot u\big]^2 \phi(u)^2 \, dA(u) dA(z) + O(\rho^{(d-3)/2}).
\end{align}
Thus, for $d>2$, whenever the integral in \eqref{eq_lll} does not vanish, $\var[\tr_\rho(f)]$ is of order $\rho^{d/2-1}$ for large $\rho$. On the other hand, $\mathbb{E}[\tr_\rho(|f|)]=\rho^{d/2} \cdot \phi(0) \cdot \int |f|$. Thus, \eqref{eq:variance_cond} is satisfied for some $\delta>0$, and Soshnikov's asymptotic normality theorem is applicable.
\end{rem}

Similarly, we now note that if we relax the assumptions of Corollary \ref{cor:main} to allow for non-reproducing kernels, then each such kernel leads to smooth observables with variance of order at least $\rho$ (and therefore falls within the scope of Soshnikov's normality theorem).

\begin{prop}\label{prop_non}
Let $K: \CC \times \CC \to \CC$ be a continuous Hermitian symmetric kernel acting boundedly on $L^2(\CC, dA)$ and satisfying \eqref{A} with $\phi$ continuous and bounded. Instead of assuming the reproducing property, suppose only that the spectrum of $T_K$, the integral operator with kernel $K$, satisfies 
\begin{align}\label{C}
\sigma(T_K) \subset [0,1],
\end{align}
so that there exists a unique DPP associated with $K$. Define the dilated kernels $K_\rho$ by \eqref{eq_dilated}. 

Suppose that $K$ is not reproducing (i.e., $T_K \not= T_K^2$). Then, for every non-zero compactly supported $C^3$ test function $f$, 
\begin{align}\label{eq_C}
\liminf_{\rho\to\infty} \frac{1}{\rho} \var[\tr_\rho(f)] >0.
\end{align}
\end{prop}
\begin{proof}
First note that the operator $T_K$ is locally trace class by assumption \eqref{A}. This fact along with \eqref{C} guarantee the existence of a unique DPP associated with $K$ (see Remark \ref{rem_ltc}). 

Let $f$ be a non-zero compactly supported $C^3$ test function. Taking into account that $K$ is not reproducing, we write $\var[\tr_\rho(f)]$ as
\begin{align*}
&\frac12 \int_{\CC^2} \big[f(z)-f(w)\big]^2 |K_\rho(z,w)|^2 \, dA(z,w) + \int_\CC f(z)^2 \Big[K_\rho(z,z)-\int_\CC |K_\rho(z,w)|^2 \,dA(w) \Big] \,dA(z)
\\
&\qquad \geq  \int_\CC f(z/\sqrt{\rho})^2 \Big[K(z,z)-\int_\CC |K(z,w)|^2 \,dA(w) \Big] \,dA(z)
\\
&\qquad=\rho \cdot \int_\CC f^2 \,dA \cdot \big[\phi(0)-\int_\CC \phi^2\,dA\big],
\end{align*}
where the first line is a general fact \cite[Equation 5]{soshnikov2002gaussian}, and the last equality follows from \eqref{A}. 

Since $f$ is non-zero, to prove \eqref{eq_C}, we only need to observe that for non-reproducing $K$, $\phi(0)-\int_\CC \phi^2\,dA \not=0$. Suppose on the contrary that $\phi(0)=\int_\CC \phi^2\,dA$. Then \[K(z,z)=\int_\CC K(z,w) K(w,z) \,dA(w), \qquad z \in \CC.\] This means that the diagonals of the integral kernels of the operators $T_K$ and $T_K^2$ coincide. On the other hand, by \eqref{C}, $T_K -T_k^2$ is a positive operator. Thus $\mathrm{trace}(T_K-T_K^2)=0$ and $T_K=T_K^2$ --- see, e.g., \cite[Theorem 2.1]{simon} --- which contradicts the assumption that $K$ is not reproducing.
\end{proof}

\end{document}